\documentclass[11pt,reqno,a4paper]{amsart}

\usepackage{etoolbox}
\usepackage{amsmath}
\usepackage{enumerate}
\usepackage{amssymb}
\usepackage[initials]{amsrefs}
\usepackage{graphicx,latexsym,amsbsy,amscd,amsthm,epsfig,verbatim}
\usepackage{amsfonts}
\usepackage{comment}
\usepackage{courier} 
\usepackage[T1]{fontenc}                                                    
\usepackage{hyperref}
\usepackage{slashed}
\usepackage{tikz-cd}

\usepackage{todonotes}

\numberwithin{equation}{section}

\newcommand{\R}{\mathbb{R}}

\newcommand\subsetsim{\mathrel{%
\ooalign{\raise0.2ex\hbox{$\subset$}\cr\hidewidth\raise-0.8ex\hbox{\scalebox{0.9}{$\sim$}}\hidewidth\cr}}}

\newtheorem{theorem}{Theorem}[section]
\newtheorem{corollary}[theorem]{Corollary}
\newtheorem{proposition}[theorem]{Proposition}
\newtheorem{lemma}[theorem]{Lemma}

\theoremstyle{definition}
\newtheorem{definition}[theorem]{Definition}

\newtheorem{remark}[theorem]{Remark}

\patchcmd{\subsection}{-.5em}{.5em}{}{}
\patchcmd{\subsubsection}{-.5em}{.5em}{}{}

\title[An $\ell_1$-norm-mass inequality for complete manifolds]{An $\ell_1$-norm-mass inequality for complete manifolds}
\author{Caterina Campagnolo and Shi Wang}

\address{Departamento de Matem\'aticas, Universidad Aut\'onoma de Madrid, Madrid, Spain}
\email{caterina.campagnolo@uam.es}

\address{Institute of Mathematical Sciences, ShanghaiTech University, Shanghai, China} 
\email{shiwang.math@gmail.com}

\begin{document}
\begin{abstract}
We generalize an inequality of Besson-Courtois-Gallot about volume and simplicial volume of closed manifolds to the $\ell_1$-norm of all the homology classes of complete manifolds. The inequality involves the critical exponent of the fundamental group of the manifold and the mass of the homology classes.
\end{abstract}
\maketitle

\section{Introduction}
Given a topological space $M$ and any singular homology class $\alpha\in H_k(M,\R)$, its \emph{Gromov norm} or \emph{$\ell_1$-norm} is defined as
\[\|\alpha\|_1:=\inf\left\{\sum_{i=1}^\ell |a_i|\mid \sum_{i=1}^\ell a_i\sigma_i \text{ is a singular cycle representing }\alpha \right\},\]
which measures a certain topological complexity of the homology class. In particular, if $M$ is a closed, connected, oriented manifold, then the Gromov norm of its fundamental class is called the \emph{simplicial volume} of $M$, denoted by $\|M\|$.

Gromov first introduced this notion in his proof of Mostow rigidity \cite{Munkholm_1980}, and observed its close relation to the Riemannian volume in his seminal paper {\cite{Gromov_1982}. He proved that:

\begin{theorem}\cite{Gromov_1982}
	If a closed connected oriented Riemannian $n$-manifold $M$ has Ricci curvature bounded below by $-(n-1)$, then
	\[\|M\|\leq(n-1)^nn!\operatorname{vol}(M).\] 
\end{theorem}
As a consequence, he gives the first non-trivial topological lower bound for the minimal volume of a smooth manifold $M$. Since then, it has been a recurrent theme in the field to obtain constraints and relations between volume and simplicial volume of a manifold. In 1991, Besson, Courtois, and Gallot \cite{besson-courtois-gallot} improved the upper bound of Gromov and proved:

\begin{theorem}\cite{besson-courtois-gallot}\label{thm:BCG91} Any closed connected oriented Riemannian $n$-manifold $M$ satisfies
	\[\|M\|\leq\frac{(n-1)^nn!}{n^{n/2}}\operatorname{minvol}(M),\]
	where
	\[\operatorname{minvol}(M)=\inf\{\operatorname{vol}_g(M)\mid g \text{ is a complete Riemannian metric on $M$ whose}\] \[\text{sectional curvature satisfies }|K|\leq 1\}.\]
\end{theorem}
It is natural to extend the above inequalities to lower degree homology classes, after replacing the volume by a proper geometric notion of ``volume of a homology class''. It is shown by Gromov that:

\begin{theorem} \cite[Section 2.5]{Gromov_1982}\label{thm:Gromov_mass}
	Let $M$ be a complete Riemannian $n$-manifold with Ricci curvature bounded below by $-(n-1)$, and let $\alpha\in H_k(M,\mathbb R)$ be any homology class. Then
	\[\|\alpha\|_1\leq (n-1)^kk!\cdot \mathrm{mass}(\alpha).\]
\end{theorem}

\begin{remark}
The definition of $\mathrm{mass}(\alpha)$ is given in Section \ref{background} (Definition 
\ref{def:mass}). However, it is worth mentioning here that it is a finer notion than the ``minimal volume'' of $\alpha$. More precisely, if $\sum_{i}^\ell a_i\sigma_i$ is any piecewise smooth cycle representing $\alpha$, then we have
	\[\mathrm{mass}(\alpha)\leq \sum_{i=1}^\ell |a_i|\mathrm{vol}_k(\sigma_i).\]
Here $\mathrm{vol}_k(\sigma_i)$ denotes the natural $k$-dimensional volume of the image of $\sigma_i$ in $M$. A more precise definition can be given by the following: let $\sigma\colon \Delta^k\rightarrow M$ be a $C^1$ $k$-simplex (the situation of piecewise smooth simplex can be easily generalized) where $\Delta^k$ is the standard Euclidean $k$-simplex with any chosen Riemannian metric. We denote the corresponding volume form by $dV_{\Delta}$. Then
\[\mathrm{vol}_k(\sigma):=\int_{\Delta^k} \operatorname{Jac}_k \sigma dV_\Delta,\]
where the definition of the $k$-Jacobian function is given by
\[\operatorname{Jac}_k(\sigma)(x)=\sup \|d\sigma (e_1)\wedge d\sigma (e_2)\wedge...\wedge d\sigma (e_k)\|,\quad\forall x\in \Delta,\]
where the supremum is taken over all orthonormal $k$-frames $\{e_1, ..., e_k\}$ on $T_x\Delta$, and the
norm is induced by the Riemannian inner product at $T_{\sigma(x)}M$. It is clear that the definition of $\mathrm{vol}_k$ is independent of the choice of the Riemannian metric on $\Delta^k$.
\end{remark}

The main purpose of this paper is to follow the general approach of Besson-Courtois-Gallot \cite{besson-courtois-gallot}, and sharpen the above mass inequality of Gromov, where the linear constant will now depend on the critical exponent of $M$.

Given a complete connected Riemannian manifold $M$, write $\Gamma=\pi_1(M)$ for its fundamental group and $\widetilde M$ for its Riemannian universal cover. Let $\rho$ be the distance function on $\widetilde{M}$ and $O\in \widetilde{M}$ be a basepoint. The critical exponent of $M$ (or of $\Gamma$ associated to its action on $\widetilde M$) is defined to be
\[\delta=\inf\{s\in\mathbb{R}\mid\sum_{\gamma\in \Gamma}e^{-s\rho(O,\gamma O)}<\infty\}.\]
It is clear that the definition is independent on the choice of the basepoint.

Our main result states:

\begin{theorem}\label{main theorem intro}
Let $M$ be a complete connected Riemannian $n$-manifold, and let $\delta$ be the critical exponent of $M$. Then for all $1\leq k\leq n$ and for all $\alpha\in H_k(M, \mathbb{R})$, we have
	$$\|\alpha\|_1\leq \frac{\delta^kk!}{{k}^{k/2}}\cdot \mathrm{mass}(\alpha).$$
\end{theorem}

\begin{remark}
	Besson, Courtois and Gallot prove essentially this statement for the fundamental class of \emph{closed} manifolds \cite[Theorems D and 3.16]{besson-courtois-gallot}. Note that under the condition $|K|\leq 1$ (or more generally $\mathrm{Ric}\geq -(n-1))$, we have $\delta\leq (n-1)$ according to the Bishop-Gromov inequality \cite[Lemma 7.1.3]{petersen}. Thus, our theorem recovers Besson, Courtois and Gallot's result in top degree (Theorem \ref{thm:BCG91}), and tightens Gromov's mass inequality (Theorem \ref{thm:Gromov_mass}) in all degrees.
\end{remark}
Our proof follows the strategy of Besson-Courtois-Gallot: it is based on Gromov's idea \cite[Section 2.4]{Gromov_1982} of using a \emph{smoothing operator} $\psi\colon\widetilde{M}\rightarrow \mathcal{M}_1$, an equivariant map from the universal cover of the manifold under consideration to its space of probability measures, to gain estimates on the $\ell_1$- and $\ell_\infty$-norms. We choose a particular family of smoothing operators $\psi^s=J\circ\Phi^s\colon \widetilde{M}\rightarrow \mathcal{M}_1$, study its properties and obtain the corresponding estimates. The particular choice  of $\psi^s$ makes the critical exponent of $M$ appear.

The paper is structured as follows. Section \ref{section:applications} contains corollaries and applications of our main result. In Section \ref{background} we recall useful definitions and prove the necessary lemmas towards the main theorem (Theorem \ref{main theorem intro} in this introduction and Theorem \ref{main theorem} later). Section \ref{section:proof of main theorem} is devoted to its proof.
\subsection*{Acknowledgements}
We would like to thank Christoforos Neofytidis for an early occasion to discuss this work and for suggesting Corollary \ref{cor:tight submanifolds}. We also thank Florent Balacheff, Bob Bell, Ulrich Bunke, Chris Connell, Jean-François Lafont, Clara Löh, George Raptis and Roman Sauer for helpful comments and discussions. We are grateful to the referee for their careful reading and useful suggestions.

C. C. acknowledges support from the Swiss National Science Foundation in the form of grant P400P2-191107 and support from the European Union in the form of a María Zambrano grant. S. W. thanks the Max Planck Institute for Mathematics in Bonn and Michigan State University for their hospitality, where part of this project was conducted.
\section{Applications}\label{section:applications}
\subsection{Norm vanishing results}
\begin{corollary}
	Let $M$ be a complete connected Riemannian manifold whose critical exponent $\delta_{\widetilde{M}}(\Gamma)$ is zero. Then 
	$\|\alpha\|_1=0$ for all $\alpha\in H_k(M, \mathbb{R})$, for all $k\geq 1$.
\end{corollary}

In the case where $\Gamma=\pi_1(M)$ is finitely generated, the critical exponent relates to the growth rate of the group. It is clear that for any choice of finite generating set $S$ of $\Gamma$, there exists a constant $K>0$ such that $$\delta_{\mathrm{Cay}(\Gamma, S)}(\Gamma)\leq K\cdot \delta_{\widetilde{M}}(\Gamma),$$
where $\delta_{\mathrm{Cay}(\Gamma, S)}$ is the exponential growth rate of $\Gamma$ with respect to the generating set $S$. Thus, $\delta_{\widetilde{M}}(\Gamma)=0$ implies that $\Gamma$ has subexponential growth. It follows that $\Gamma$ is amenable. It is well known since the work of Trauber that the bounded cohomology of an amenable group vanishes in all degrees at least $1$ \cite{johnson}. Hence by the duality principle and Gromov's mapping theorem \cite{Gromov_1982}, the $\ell_1$-norm vanishes identically in all degrees at least $1$ on the homology groups of a manifold with amenable fundamental group. Thus the corollary follows.

On the other hand, if $M$ is a closed Riemannian manifold whose fundamental group $\Gamma$ has subexponential growth, then $\delta_{\widetilde{M}}(\Gamma)$ vanishes since the universal cover $\widetilde{M}$ of $M$ is quasi-isometric to the Cayley graph $\mathrm{Cay}(\Gamma, S)$ of $\Gamma$ with respect to any finite generating set $S$. Thus, in this case our theorem partially recovers Trauber's result. This includes all compact nilmanifolds. Note that it is an open problem whether there exists a finitely presented group with intermediate growth.

\subsection{Bounds on submanifolds}
\begin{corollary}\label{cor:mapping statement}
Let $M$ be a complete connected Riemannian manifold and let $\delta$ be its critical exponent. Let $X$ be a path-connected topological space and $f\colon X\rightarrow M$ a continuous map that induces a surjection with amenable kernel on the level of fundamental groups. Then for every class $\alpha\in H_k(X, \mathbb{R})$, $k\geq 1$, we have 
$$\|\alpha\|_1\leq \frac{\delta^kk!}{k^{k/2}}\cdot\mathrm{mass}(f_*(\alpha)).$$
\end{corollary}
\begin{proof}
By Gromov's mapping theorem \cite{Gromov_1982} we obtain that $\|\alpha\|_1=\|f_*(\alpha)\|_1$. Hence, the $\ell_1$-norm of $\alpha$ is bounded by the same bound as $\|f_*(\alpha)\|_1$ is. We conclude by Theorem \ref{main theorem intro}. 
\end{proof}
Now, up to a factor, Thom's realization theorem ensures that rational homology classes can be represented by manifolds. This leads to the following estimate for the representing manifolds:
\begin{corollary}\label{cor:tight submanifolds}
Let $M$ be a complete connected Riemannian manifold with finitely presented fundamental group and let $\alpha\in H_k(M, \mathbb{Q})$ be a rational homology class, $k\geq 1$. Let $N$ be any oriented closed connected $k$-manifold $N$ representing the class $\alpha$, that is there exists a continuous map $f\colon N\rightarrow M$ and $r\in\mathbb{Q}\setminus \{0\}$ such that $f_\ast([N])=r\alpha$, and suppose additionally that $\pi_1(f)\colon N\rightarrow M$ is surjective with amenable kernel. Then we have 
$$\|N\|\leq |r|\cdot\frac{\delta^kk!}{k^{k/2}}\cdot\mathrm{mass}(\alpha).$$
\end{corollary}
\begin{remark}
The existence of such an $N$ representing such an $\alpha$ is ensured by Thom's realization theorem. Moreover, if $k\leq 6$ or $k=n-1$ and $\alpha$ is an integral class, we can take $r=1$. For $k\geq 4$, surgery (see for example \cite[Theorem 3.1]{crowley-loeh}) allows to chose $N$ with an isomorphism $\pi_1(N)\cong\pi_1(M)$ induced by $f$. 
\end{remark}
\begin{proof}
This is a particular case of Corollary \ref{cor:mapping statement}. By Gromov's mapping theorem \cite{Gromov_1982} we obtain that $\|N\|=|r|\|\alpha\|_1$. Hence, the simplicial volume of $N$ is bounded by the same bound as $|r|\|\alpha\|_1$ is. We conclude as above. 
\end{proof}

\subsection{Double sided inequality}
As we see, our main theorem provides a linear upper bound on the Gromov norm in terms of the mass, where the linear constant depends on the critical exponent. On the other hand, if the complete manifold $M$ has pinched negative sectional curvature, then the Gromov norm is also linearly bounded from below by the mass. More generally, this is true for any $k$-homology class whenever the manifold admits a straightening and the straightened $k$-simplices have uniformly bounded $k$-volume. Given a complete manifold $M$ and its fundamental group $\Gamma$, we recall the following definition:

\begin{definition}\label{def:straightening}
	Let $S_k\colon\Delta_k(\widetilde M)\rightarrow \Delta_k(\widetilde M)$ be a family of maps for $k\in\{0, 1,..., n\}$, where $\Delta_k(\widetilde M)$ denotes the set of all singular simplices on $\widetilde M$. We say that the family $\{S_k\}$ is a \emph{straightening} on $M$ if 
	\begin{enumerate}
		\item $S_k$ is $\Gamma$-equivariant for any $k\in \{0,1,...,n\}$, that is, for any $\sigma\colon\Delta^k\rightarrow \widetilde M$ and any $\gamma\in \Gamma$, we have $\gamma S_k(\sigma)=S_k(\gamma \sigma)$. Thus, $S_k$ descends to a self map on $\Delta^k(M)$.
		\item The maps $S_k$, $k\in \{0,1,...,n\}$, induce a chain map $\mathrm{C}_*(\widetilde M,\R)\rightarrow \mathrm{C}_*(\widetilde M,\R)$ which is $\Gamma$-equivariantly chain homotopic to the identity. Thus, the $S_k$, $k\in \{0,1,...,n\}$, descend to a chain map on $\mathrm{C}_*(M,\R)$ which is homotopic to the identity.
	\end{enumerate}
\end{definition}

The following proposition is a well-known strategy to bound a Gromov norm from below.

\begin{proposition}\label{prop:straightening}
	Suppose $M$ admits a straightening, all straightened $k$-simplices are $C^1$-smooth and their volumes are uniformly bounded by a constant $C_k>0$ (possibly depending on $M$). Then, for any $\alpha\in H_k(M,\R)$, we have the inequality
	\[\|\alpha\|_1\geq \frac{1}{C_k}\operatorname{mass}(\alpha).\]
\end{proposition}
\begin{proof}
For every $\alpha=[\sum_i a_i\sigma_i]\in H_k(M,\mathbb{R})$, we have
\begin{align*}
\mathrm{mass}(\alpha)&=\sup_{\beta,\, \mathrm{comass}(\beta)\leq 1}\int_\alpha \beta\\
&=\sup_{\beta,\, \mathrm{comass}(\beta)\leq 1}\int_{[\sum_i a_i\sigma_i]}\beta\\
&=\sup_{\beta,\, \mathrm{comass}(\beta)\leq 1}\int_{[\sum_i a_i S_k(\sigma_i)]}\beta\\
&\leq\sup_{\beta, \,\mathrm{comass}(\beta)\leq 1}\sum_i|a_i|\mathrm{comass}(\beta)\mathrm{vol}_k(S_k(\sigma_i))\\
&\leq C_k\cdot \sum_i|a_i|,
\end{align*}
where the third equality uses $(2)$ of Definition \ref{def:straightening}.
Finally, after taking the infimum over all representatives of $\alpha$, we obtain 
$$\mathrm{mass}(\alpha)\leq C_k\cdot \|\alpha\|_1,$$
hence the proposition follows.
\end{proof}

When $M$ has non-positive curvature, there are known results where the conditions of Proposition \ref{prop:straightening} are satisfied. We summarize them below: Proposition \ref{prop:straightening} holds if
\begin{itemize}
	\item \cites{Gromov_1982, Thurston_1979} $M$ is real hyperbolic and $k\geq 2$, where $C_k$ is the maximal volume of an ideal $k$-simplex in $\mathbb H^n$. The constant $C_k$ is explicitly estimated in \cite{Haagerup-Munkholm}.
	\item \cite{Inoue-Yano} $M$ has sectional curvature bounded away from zero and $k\geq 2$, where $C_k$ is explicitly estimated in terms of the curvature bound.
	\item \cites{Lafont-Wang, Wang} $M$ is a locally symmetric space of non-compact type (without certain small factors) and $k\geq \operatorname{srk}(\widetilde M)+2$, where $\operatorname{srk}(\widetilde M)$ is explicitly computed in \cite[Table 1]{Wang}, and $C_k$ depends on $\widetilde M$.
	\item \cite{Connell-Wang} $M$ satisfies $\operatorname{Ric}_{\ell+1}<0$ and $k\geq 4\ell$, and $C_k$ depends on $\widetilde M$.
\end{itemize}
Note that the first two results use the geodesic straightening, and the last two results use the barycentric straightening introduced in \cite{Lafont-Schmidt}.
\begin{corollary}
	If the conditions of Proposition \ref{prop:straightening} are satisfied, then for every $\alpha\in H_k(M,\R)$, $k\geq 1$, we have double sided inequalities
	\[\frac{1}{C_k}\operatorname{mass}(\alpha)\leq \|\alpha\|_1\leq \frac{\delta^kk!}{{k}^{k/2}}\mathrm{mass}(\alpha).\]
	In particular, if
	\[\delta< \left(\frac{k^{k/2}}{C_k k!}\right)^{1/k},\]
	then any $\alpha\in H_k(M,\R)$ satisfies
    \[\|\alpha\|_1=\mathrm{mass}(\alpha)=0.\]
\end{corollary}
\begin{remark}
When $M$ is hyperbolic, we can show that
\[\left(\frac{k^{k/2}}{C_k k!}\right)^{1/k}<k-1.\]
Now by the vanishing homology theorem of Kapovich \cite{Kapovich}, if $\delta<k-1$ the only possible non-trivial homology classes come from the cusps of $M$, whose Gromov norm is automatically zero. The second part of the above Corollary thus follows immediately. On the other hand, by a theorem of Mineyev \cite{mineyev}, if $\pi_1(M)$ is Gromov hyperbolic, then any non-trivial $k$-homology class (where $k\geq 2$) has positive Gromov norm. So together with the above corollary, it is plausible to obtain similar homology vanishing results for certain groups (e.g. if the comparison map is known to be surjective in certain degrees). We suspect there are interesting applications when $M$ has non-positive curvature.
\end{remark}

\section{Background}\label{background}
We collect here the notation, the definitions and the preparatory results for the proof of Theorem \ref{main theorem intro}. The proof itself will be given in Section \ref{section:proof of main theorem} (Theorem \ref{main theorem}).

Let $M$ be a complete connected Riemannian manifold with a Riemannian metric $g$. Let $\Gamma$ be its fundamental group, $\widetilde{M}$ its universal cover. We write $\rho\colon \widetilde{M}\times\widetilde{M}\rightarrow \mathbb{R}_{\geq 0}$ for the distance on $\widetilde{M}$ induced by a lift of $g$, and $\rho_x$ for the function $\rho(x, -)$, $x\in \widetilde{M}$. Recall the following definition of the critical exponent.

\begin{definition}\label{def:critical_exp}
Let $M$ be a connected Riemannian manifold. With the above notation, the \emph{critical exponent} of $M$ (or of $\Gamma$ associated to its action on $\widetilde M$) is defined as
	\[\delta=\delta_{\widetilde M}(\Gamma):=\inf\{s\in \mathbb{R}\mid\sum_{\gamma\in \Gamma}e^{-s\rho(O,\gamma O)}<\infty\},\]
	where $O\in \widetilde{M}$ is any chosen basepoint in $\widetilde{M}$. By the triangle inequality, it is clear that $\delta$ is independent of the choice of $O$.
\end{definition}

Let $o\in M$ be a chosen basepoint in $M$. Let $\mu$ be a finite positive measure on $M$ chosen to decay fast enough so that
\begin{equation}\label{eq:fastdecay}
	\int_{M}e^{(\delta+1)d(o,x)}d\mu(x)<\infty,
\end{equation}
where $d$ is the distance function on $M$ associated with $g$. By the triangle inequality again, it is clear that the condition on $\mu$ does not depend on the basepoint $o$. We can further choose $\mu$ to be absolutely continuous with respect to the volume measure on $M$. We denote by $\widetilde \mu$ the lift of $\mu$ to $\widetilde M$: it is a positive measure on $\widetilde{M}$, absolutely continuous with respect to the lift of the volume measure induced by $g$, and $\Gamma$-invariant.

\begin{lemma}\label{lem:L2_convergent}
	Following the notations above, if $M$ is a complete connected Riemannian manifold whose critical exponent equals $\delta$, then for every $x\in \widetilde{M}$, the function $\varphi_x^s(y)=e^{-s\rho_x(y)}$ belongs to $L^2(\widetilde M,\widetilde \mu)$ whenever $\delta/2<s\leq(\delta+1)/2$.
\end{lemma}

\begin{proof}
	By triangle inequality, it suffices to show that the function $\varphi_O^s$ is in $L^2(\widetilde M,\widetilde \mu)$. We choose a fundamental domain $\mathfrak F\subset \widetilde{M}$. Then the square of the $L^2$-norm of $\varphi_O^s$
	 is given by
	\begin{align*}
		\|\varphi_O^s\|^2_{L^2}&=\int_{\widetilde M}e^{-2s\rho_O(x)}d\widetilde\mu(x)\\
		&=\sum_{\gamma\in \Gamma}\int_{\gamma \mathfrak F} e^{-2s\rho_O(x)}d\widetilde\mu(x)\\
		&=\sum_{\gamma\in \Gamma}\int_{\mathfrak F} e^{-2s\rho_{\gamma^{-1}O}(x')}d\widetilde\mu(x')\quad \textit{ substituting }x=\gamma x'\\
		&=\int_{\mathfrak F} \sum_{\gamma\in \Gamma}  e^{-2s\rho_{\gamma^{-1}O}(x')}d\widetilde\mu(x')
	\end{align*}
where the last equality uses the dominated convergence theorem, provided the integral of the function
$$\int_{\mathfrak F} \sum_{\gamma\in \Gamma}  e^{-2s\rho_{\gamma^{-1}O}(x')}d\widetilde\mu(x')$$ is finite, which we are going to show in the next few steps. Let
\[F^s(x)=\sum_{\gamma\in \Gamma}  e^{-2s\rho_{\gamma^{-1}O}(x)}.\]
By the definition of the critical exponent together with the triangle inequality, we know $F^s(x)<\infty$ whenever $s>\delta/2$. Also, it is clear that $F^s(x)$ is $\Gamma$-invariant, hence it descends to a function $G^s$ on $M$. Moreover, for each $x,x'\in \widetilde M$, when estimating $F^s(x)/F^s(x')$, by possibly replacing $x, x'$ by elements in the same $\Gamma$-orbit, we may assume $\rho(x,x')=d(\pi(x),\pi(x'))$ where $\pi\colon\widetilde M\rightarrow M$ denotes the universal covering map. Thus,
\begin{align*}
	\frac{F^s(x)}{F^s(x')}&=\frac{\sum_{\gamma\in \Gamma}  e^{-2s\rho_{\gamma^{-1}O}(x)}}{\sum_{\gamma\in \Gamma}  e^{-2s\rho_{\gamma^{-1}O}(x')}}\\
	&\leq \frac{e^{2s\rho(x,x')}\sum_{\gamma\in \Gamma}  e^{-2s\rho_{\gamma^{-1}O}(x')}}{\sum_{\gamma\in \Gamma}  e^{-2s\rho_{\gamma^{-1}O}(x')}}\\
	&=e^{2s\rho(x,x')}.
\end{align*}
By setting $x'=O$, we obtain that
\[F^s(x)\leq e^{2s\rho(x,O)}F^s(O),\]
and so
\[G^s(\pi(x))\leq e^{2sd(\pi (x),\pi(O))}G^s(\pi(O)).\]
Therefore, we can further estimate
\begin{align*}
	\|\varphi^s_O\|^2_{L^2}&=\int_{\mathfrak F} F^s(x) d\widetilde\mu(x)\\
	&=\int_M G^s(z)d\mu(z)\\
	&\leq \int_M e^{2sd(z,\pi(O))}G^s(\pi(O))d\mu(z)\\
	&<\infty,
\end{align*}
where the last inequality uses \eqref{eq:fastdecay} provided $s\leq (\delta+1)/2$.
\end{proof}
\begin{definition}\label{def:Phi_s}
We define the following maps $\Phi^s\colon\widetilde{M}\rightarrow S^\infty\subset L^2(\widetilde{M}, \widetilde{\mu})$:
$$x\longmapsto \left\{y\longmapsto \frac{\varphi_x^s(y)}{\|\varphi_x^s\|_{L^2}}\right\},$$
where $S^\infty$ denotes the unit sphere of $L^2(\widetilde{M}, \widetilde{\mu})$, $s$ satisfies $\frac{\delta+1}{2}\geq s>\frac{\delta}{2}$ and $\delta$ is the critical exponent of  $\Gamma$ for its action on $\widetilde{M}$.
Under this assumption, the $\Phi^s(x)$ are $L^2$-functions by Lemma \ref{lem:L2_convergent}. They are of norm $1$ by definition and it is straightforward to check that they are $\Gamma$-equivariant.
\end{definition}
We show an estimate on the $k$-Jacobian of $\Phi^s$, which will be useful in the proof of the main theorem:
\begin{lemma}\label{lem:k-Jac}
	Given a complete connected Riemannian $n$-manifold $M$ whose critical exponent is $\delta$, let $\Phi^s$ be the map defined above. For any integer $2\leq k\leq n$, we have
	\[|\operatorname{Jac}_k\Phi^s|\leq \left(\frac{s^2}{k}\right)^{k/2}.\]
\end{lemma}

\begin{proof}
	Fix any $x\in \widetilde M$ and any $X\in T_x\widetilde M$. Recall that $\varphi_x^s=e^{-s\rho_x(\cdot)}$ and $\Phi^s(x)=\varphi_x^s/\|\varphi_x^s\|_{L^2}$ belong to $L^2(\widetilde M,\widetilde \mu)$ by Lemma \ref{lem:L2_convergent}. We will write $\|\cdot\|_{L^2}=\|\cdot\|$ for short. We compute
	\begin{align*}
		d\Phi^s_x(X)=d(\frac{\varphi_x^s}{\|\varphi_x^s\|})(X)=\frac{d\varphi_x^s(X)\cdot\|\varphi_x^s\|-\varphi_x^s\cdot d\|\varphi_x^s\|(X)}{\|\varphi_x^s\|^2}.
	\end{align*}
    Note that $L^2(\widetilde M,\widetilde \mu)$ is an affine space, so its tangent spaces can be naturally identified with itself. Under this identification,  it is clear that $d\|\varphi_x^s\|(X)\in L^2(\widetilde M,\widetilde \mu)$ is a constant function, so $\varphi_x^s\cdot d\|\varphi_x^s\|(X)$ is parallel to $\varphi_x^s$. Since $\varphi_x^s/\|\varphi_x^s\|\in S^\infty\subset L^2(\widetilde M,\widetilde \mu)$, we know $d(\frac{\varphi_x^s}{\|\varphi_x^s\|})(X)$ is orthogonal to $\varphi_x^s$. Thus by the pythagorean theorem, we have
    \[\|d\Phi^s_x(X)\|^2\leq \frac{\|d\varphi_x^s(X)\|^2}{\|\varphi_x^s\|^2}.\]
    Suppose that, under a proper choice of orthonormal bases at $T_x\widetilde M$ and at its image space in $T_{\Phi^s(x)}(L^2(\widetilde M,\widetilde \mu))$, $d\Phi_x^s$ is diagonalized to have eigenvalues $0\leq |\lambda_n|\leq \cdots\leq |\lambda_1|$. Then we know that $|\operatorname{Jac}_k\Phi^s|=\prod_{i=1}^k|\lambda_i|$. Thus by the geometric-arithmetic mean inequality and the Cauchy-Schwarz inequality, provided $k\geq 2$ we obtain that
    \begin{align*}
    	\prod_{i=1}^k|\lambda_i|&\leq \left(\frac{ \sum_{i=1}^k|\lambda_i|}{k}\right)^k\\
    	&\leq \left(\frac{(\sum_{i=1}^k|\lambda_i|^2)^{1/2}\cdot \sqrt k}{k}\right)^k\\
    	&\leq \left(\frac{\operatorname{tr}((d\Phi^s_x)^*\circ d\Phi^s_x)}{k}\right)^{k/2}.
    \end{align*}
Finally we can estimate $\operatorname{tr}((d\Phi^s_x)^*\circ d\Phi^s_x)$ by choosing an orthonormal basis $\{e_1,..., e_n\}$ on $T_x\widetilde M$:
\begin{align*}
	\operatorname{tr}((d\Phi^s_x)^*\circ d\Phi^s_x)&= \sum_{i=1}^n \|d\Phi_x^s(e_i)\|^2\\
	&\leq \frac{\sum_{i=1}^n\|d\varphi_x^s(e_i)\|^2}{\|\varphi_x^s\|^2}\\
	&=\frac{\sum_{i=1}^n s^2\int_{\widetilde M} e^{-2s\rho_x(z)}d\rho_{x,z}^2(e_i)d\widetilde \mu(z)}{\|\varphi_x^s\|^2}\\
	&=s^2.
\end{align*}
Therefore, combining the above inequalities, we have
\[|\operatorname{Jac}_k\Phi^s|\leq \left(\frac{s^2}{ k}\right)^{k/2}.\hfill\qedhere\]
\end{proof}

\begin{definition}
\begin{enumerate}
\item We denote by $\mathcal{M}$ the vector space of bounded measures on $\widetilde{M}$. It is endowed with the \emph{total variation norm}: by the Jordan decomposition theorem, every measure $\mu$ can be uniquely decomposed into two positive measures, its positive and its negative part $\mu_+$ and $\mu_-$, so that $\mu=\mu_+-\mu_-$. The total variation of $\mu$ is then
$$|\mu|=\int_{\widetilde{M}}\mu_+ +\int_{\widetilde{M}}\mu_-\geq 0.$$
\item We denote by $\mathcal{M}_1\subset\mathcal{M}$ the subspace of probability measures on $\widetilde{M}$.
\end{enumerate}
\end{definition}
We now define a map
\begin{align*}
J\colon S^\infty\subset L^2(\widetilde{M}, \widetilde{\mu})\longrightarrow &\mathcal{M}_1\\
f\longmapsto &f^2\widetilde{\mu}.
\end{align*}

\begin{lemma}\label{map J 2-lipschitz}
The map $J$ is $2$-Lipschitz, where the distance on $\mathcal{M}_1$ is inherited from $\mathcal{M}$.
\end{lemma}

\begin{proof}
We compute
	\begin{align*}
		d_\mathcal M(J(f),J(g))= d_\mathcal M(f^2\widetilde{\mu},g^2\widetilde{\mu})&\leq \int_{\widetilde M}|f^2-g^2| d\widetilde{\mu}\\
		&\leq \left(\int_{\widetilde M}|f-g|^2 d\widetilde{\mu}\right)^{1/2} \left(\int_{\widetilde M}|f+g|^2 d\widetilde{\mu}\right)^{1/2}\\
		&=d_{L^2}(f,g) \cdot \left(\int_{\widetilde M}|f+g|^2 d\widetilde{\mu}\right)^{1/2}.
	\end{align*}
We can also estimate
	\begin{align*}
	\int_{\widetilde M}|f+g|^2 d\widetilde{\mu}&=  	\int_{\widetilde M}f^2 d\widetilde{\mu}+	\int_{\widetilde M}g^2 d\widetilde{\mu}+	2\int_{\widetilde M}fg d\widetilde{\mu}\\
	&\leq 2+2\left(\int_{\widetilde M}f^2 d\widetilde{\mu}\right)^{1/2}\left(\int_{\widetilde M}g^2 d\widetilde{\mu}\right)^{1/2}\\
	&=4.
\end{align*}
Thus by combining the two inequalities, we obtain that $J$ is $2$-Lipschitz.
\end{proof}

\begin{definition}\label{def:mass}
We recall a few useful definitions below.
\begin{enumerate}
\item Let $\omega$ be a $k$-differential form on the manifold $M$. Its \emph{comass} is defined as
$$\mathrm{comass}(\omega)=\sup_{p\in M}\left\{\sup\left\{\omega_p(x_1, ..., x_k) \,\mid\, x_1, ..., x_k\in T_pM \mbox{ unit vectors}\right\}\right\}.$$
\item Let $\beta\in H^k(M, \mathbb{R})$ be a cohomology class. By de Rham's theorem, it corresponds to a de Rham class of closed differential $k$-forms $\omega\in \Omega^k(M)$. The \emph{comass} of $\beta$ is by definition
$$\mathrm{comass}(\beta)=\inf_{[\omega]=\beta}\mathrm{comass(\omega)}.$$
Note that for $k=0$, $\mathrm{comass}(\beta)=\|\beta\|_\infty$.
\item Let $\alpha\in H_k(M, \mathbb{R})$ be a homology class. Its \emph{mass} is defined as
$$\mathrm{mass}(\alpha)=\sup\left\{\int_\alpha \omega \,\mid\,\omega \text{ closed differential $k$-form on $M$, } \mathrm{comass}(\omega)\leq 1\right\}.$$
Note that for $k=0$, $\mathrm{mass}(\alpha)=\|\alpha\|_1$.
\item Let $a$ be an alternating $(k+1)$-form on a vector space $V$. Its \emph{sup norm} is defined as 
$$\|a\|_\infty=\sup_{v_0, v_1, ..., v_k\in V}a(v_0, v_1,..., v_k).$$
\end{enumerate}
\end{definition}
Recall that:
\begin{lemma}\label{volume simplex}
The Euclidean volume of the $k$-simplex with vertices $0$ and the canonical basis vectors $e_1, ..., e_k$ in $\mathbb{R}^k$ is $\frac{1}{k!}$.
\end{lemma}
No proof of the next lemma is given in \cites{Gromov_1982, besson-courtois-gallot}, so we give one here for the reader's convenience.
\begin{lemma}[\cite{besson-courtois-gallot}, p. 439 and \cite{Gromov_1982}, p. 33]\label{alternating and differential form}
Let $a$ be an alternating $(k+1)$-form on $\mathcal{M}$ and let $\widetilde{a}$ be the differential $k$-form on $\mathcal{M}_1$ defined by $\widetilde{a}_\mu(\mu_1, ..., \mu_k)=k!\cdot a(\mu, \mu_1, ..., \mu_k)$, where $\mu\in\mathcal{M}_1$ and $\mu_1, ..., \mu_k\in\mathcal{M}$ are elements of the tangent space to $\mathcal{M}_1$ at $\mu$, that is $\mu_i(\widetilde{M})=0$ for all $1\leq i\leq k$. Then 
\begin{align}\label{integral on simplex}
\int_{\Delta(\mu_0, ...,\mu_k)}\widetilde{a}=a(\mu_0, ...,\mu_k).
\end{align}
Moreover, if $\bar{a}$ denotes the restriction of $a$ to $\mathcal{M}_1$, then 
\begin{align}\label{norm equality}
\mathrm{comass}(\widetilde{a})=k!\cdot\|\bar{a}\|_\infty.
\end{align}
\end{lemma}
\begin{proof}
We parametrize the $k$-simplex $\Delta(\mu_0, ..., \mu_k)\subset \mathcal{M}_1$ with vertices $\mu_0, ..., \mu_k$ by the domain $$T=\{(x_1, ..., x_k)\in \mathbb{R}^k\mid \sum_{i=1}^kx_i\leq 1, x_i\geq 0\,\forall\, 1\leq i\leq k\}\subset \mathbb{R}^k$$ via the map
\begin{align*}
\phi\colon T\longrightarrow &\Delta(\mu_0, ..., \mu_k)\\
 (x_1, ..., x_k) \longmapsto& (1-\sum_{i=1}^kx_i)\mu_0+x_1\mu_1+...+x_k\mu_k.
\end{align*}
Notice that this expression has indeed norm $1$ in the total variation norm.
We chose the coordinates $\{\mu_0, \mu_1, ..., \mu_k\}$ for the subspace of $\mathcal{M}$ containing $\Delta(\mu_0, ..., \mu_k)$. In these coordinates, the map $\phi$ becomes
\begin{align*}
\phi\colon T\longrightarrow &\Delta(\mu_0, ..., \mu_k)\subset \mathbb{R}^{k+1}\\
 (x_1, ..., x_k) \longmapsto& (1-\sum_{i=1}^k x_i, x_1, ..., x_k), 
\end{align*}
so that its differential at $(x_1, ..., x_k)$ is
$$d\phi_{(x_1, ..., x_k)}=\left(\begin{array}{cccccc}
-1 & -1&-1&\dots& -1\\
1 & 0&0  &\dots &0\\
0 & 1&0  &\dots &0\\
\dots&\dots&\dots&\dots&\dots\\
\dots&\dots&\dots&\dots&\dots\\
0 & 0& \dots&0 & 1
	\end{array}\right).
	$$
We notice that the form $\widetilde{a}$ is constant on $\Delta(\mu_0, ..., \mu_k)$: indeed, let $\mu=\sum_{i=0}^k x_i\mu_i, \mu'=\sum_{i=0}^k x'_i\mu_i\in\Delta(\mu_0, ..., \mu_k)$ and $v_1, ..., v_k$ be tangent vectors to $\Delta(\mu_0, ..., \mu_k)$, that is they are of the form $v_j=\sum_{i=0}^k z_i^j\mu_i$ with $\sum_{i=0}^k z_i^j=0$. Then
\begin{align*}
\widetilde{a}_\mu(v_1, ..., v_k)-\widetilde{a}_{\mu'}(v_1, ..., v_k)&=k!\cdot a(\mu, v_1, ..., v_k)-k!\cdot a(\mu', v_1, ..., v_k)\\
 &=k!\cdot a(\mu-\mu', v_1, ..., v_k)\\
 &=k!\cdot a(\sum_{i=0}^k (x_i-x'_i)\mu_i, \sum_{i=0}^k z_i^1\mu_i, ..., \sum_{i=0}^k z_i^k\mu_i).
\end{align*}
Now all $k+1$ argument vectors in the above equation belong to the $k$-dimensional subspace of $\left\langle\{\mu_0, \mu_1, ..., \mu_k\}\right\rangle$ with coordinates summing to zero. Thus they are linearly dependent. Therefore the multilinear form $a$ vanishes on them. Consequently
$$\widetilde{a}_\mu(v_1, ..., v_k)=\widetilde{a}_{\mu'}(v_1, ..., v_k).$$
The form $\phi^*\widetilde{a}$ is a $k$-differential form on a $k$-dimensional space. The previous argument implies that it is a constant multiple of the standard volume form $dx_1\wedge...\wedge dx_k$. To determine the multiple, we evaluate $\phi^*\widetilde{a}$ on a well-chosen point. Let us write $x_0=1-\sum_{i=1}^k x_i$ and $\mu=\sum_{i=0}^k x_i\mu_i$.
\begin{align*}
\phi^*\widetilde{a}_{(x_1, ..., x_k)}(\frac{\partial}{\partial x_1}, ...,\frac{\partial}{\partial x_k}) &=\widetilde{a}_{\phi(x_1, ..., x_k)}d\phi(\frac{\partial}{\partial x_1}, ...,\frac{\partial}{\partial x_k})\\
&=\widetilde{a}_{\mu}(\mu_1-\mu_0, \mu_2-\mu_0, ..., \mu_k-\mu_0)\\
 &=k!\cdot a(\mu, \mu_1-\mu_0, ...,\mu_k-\mu_0)\\
 &=k!\cdot\left[x_0a(\mu_0,\mu_1, ...,\mu_k)-x_1a(\mu_1, \mu_0, ..., \mu_k)-...-x_ka(\mu_k, \mu_1, ..., \mu_0)\right]\\
 &=k!\cdot\sum_{i=0}^kx_ia(\mu_0, ..., \mu_k)=k!\cdot a(\mu_0, ..., \mu_k).
\end{align*}
Therefore $\phi^*\widetilde{a}=k!\cdot a(\mu_0, ..., \mu_k)dx_1\wedge...\wedge dx_k$. We now can compute the integral:
\begin{align*}
\int_{\Delta(\mu_0, ..., \mu_k)}\widetilde{a}&=\int_T\phi^*\widetilde{a}\\
&=\int_Tk!\cdot a(\mu_0, ..., \mu_k)dx_1\wedge...\wedge dx_k\\
&=k!\cdot a(\mu_0, ..., \mu_k)\int_Tdx_1\wedge...\wedge dx_k=k!\cdot a(\mu_0, ..., \mu_k)\frac{1}{k!},
\end{align*}
where the last equality follows from Lemma \ref{volume simplex}. This shows Claim \eqref{integral on simplex}.

The norm equality \eqref{norm equality} is obtained as follows:
\begin{align*}
\|\bar{a}\|_\infty&=\sup_{\mu_0, \mu_1, ..., \mu_k\in \mathcal{M}_1}a(\mu_0, ...,\mu_k)\\
	&=\frac{1}{k!}\sup_{\mu_0, \mu_1, ..., \mu_k\in \mathcal{M}_1}\widetilde{a}_{\mu_0}(\mu_1, ..., \mu_k)\\
	&=\frac{1}{k!}\cdot\mathrm{comass}(\widetilde{a}).
	\hfill\qedhere
\end{align*}
\end{proof}
\section{Proof of the main theorem}\label{section:proof of main theorem}
\begin{theorem}[Cf. \cite{besson-courtois-gallot}, Theorems D and 3.16 and \cite{Gromov_1982}, Section 2.5]\label{main theorem}
Let $M$ be a complete connected Riemannian $n$-manifold and $\delta$ be its critical exponent.
\begin{enumerate}
\item \label{statement 1}For all $1\leq k\leq \mathrm{dim}(M)$, for all $\beta\in H^k(M, \mathbb{R})$, we have
$$\mathrm{comass}(\beta)\leq \frac{\delta^kk!}{{k}^{k/2}}\cdot\|\beta\|_\infty.$$
The right hand side is $+\infty$ if $\|\beta\|_\infty=\infty$ and $\delta=0$.
\item \label{statement 2} For all $1\leq k\leq \mathrm{dim}(M)$, for all $\alpha\in H_k(M, \mathbb{R})$, we have
$$\|\alpha\|_1\leq \frac{\delta^kk!}{{k}^{k/2}}\cdot\mathrm{mass}(\alpha).$$
\end{enumerate}
\end{theorem}
\begin{remark}
\begin{enumerate}
\item Besson, Courtois and Gallot prove essentially the statement (\ref{statement 2}) for the fundamental class of \emph{closed} manifolds \cite[Theorems D and 3.16]{besson-courtois-gallot}. However we notice that being closed is not necessary: completeness is enough to prove the statement for all homology classes of $M$.
\item
Statements (\ref{statement 1}) and (\ref{statement 2}) are exactly the improvement in our more general setting of inequalities (+) and (++) of Gromov \cite[p. 36]{Gromov_1982}, corresponding to the improvement of Besson-Courtois-Gallot \cite[Theorem D]{besson-courtois-gallot} of the top-degree inequality (++) in the case of closed manifolds.
\item
Note that statement (\ref{statement 1}) for $k=1$ is uninteresting, as the sup norm of any degree $1$ cohomology class is infinite. Similarly, statement (\ref{statement 2}) for $k=1$ is trivially true because the $\ell_1$-norm always vanishes in degree $1$.
\end{enumerate}
\end{remark}
\begin{proof}
We follow Gromov and Besson-Courtois-Gallot \cites{Gromov_1982, besson-courtois-gallot}. Let $$\mathrm{C}^k_b(\widetilde{M}, \Gamma)=\left\{c\colon \widetilde{M}^{k+1}\longrightarrow \mathbb{R}\,\,\Gamma\mbox{-invariant, antisymmetric, continuous, bounded}\right\}.$$
This forms a subcomplex of the singular cochain complex of $M$, by associating to each $c$ as above the cochain
\begin{align*}
c\colon \mathrm{C}_k(M, \mathbb{R})\longrightarrow &\mathbb{R}\\
\sigma\longmapsto& c(\sigma_0, ..., \sigma_k),
\end{align*}
where $\sigma_0, ..., \sigma_k$ denote a choice of lifts to $\widetilde{M}$ of the vertices of $\sigma$. The cohomology of this complex is the bounded cohomology $H_b^\ast(M, \mathbb{R})$ of $M$ (\cite[p. 48]{Gromov_1982} and also \cite[Theorem 1.4.2 and Corollary 4.4.4]{Frigerio_Moraschini_2019}).

Recall \cite[Lemma 6.1]{Frigerio_2017} (see also \cite[Proposition F.2.2]{Benedetti_Petronio_2012}) that, for every $\alpha\in H_k(M, \mathbb{R})$, 
\begin{align}\label{duality Frigerio}
\|\alpha\|_1=\sup\left\{\langle\beta, \alpha\rangle\mid \beta\in H^k_b(M, \mathbb{R}), \|\beta\|_\infty\leq 1\right\}.
\end{align}
Consider another cochain complex $\mathrm{C}^k_b(\mathcal{M}_1, \Gamma)$ consisting of the maps
\begin{align*}
c\colon \mathcal{M}_1^{k+1}\longrightarrow \mathbb{R}
\end{align*}
that are $\Gamma$-invariant, antisymmetric, continuous, bounded, and multilinear with respect to barycentric combinations: for every $t\in [0, 1]$, for every $i\in\{0, ..., k\}$, 
$$c(\mu_0, ..., t\mu_i+(1-t)\mu_i', ..., \mu_k)=tc(\mu_0, ..., \mu_i, ..., \mu_k)+(1-t)c(\mu_0, ...,\mu_i', ..., \mu_k).$$ The coboundary operator, as usually, is given by 
\begin{align*}
\delta\colon \mathcal{M}_1^{k+1}\longrightarrow &\mathcal{M}_1^{k+2}\\
c\longmapsto &\left\{(\mu_0, ...,\mu_{k+1})\longmapsto \sum_{i=0}^{k+1}(-1)^ic(\mu_0, ..., \widehat{\mu_i}, ..., \mu_{k+1})\right\}.
\end{align*}
We denote the cohomology of this complex by $H^\ast_L(\mathcal{M}_1, \Gamma)$. It is endowed with a seminorm denoted by $\|\cdot\|_\infty$, induced on $H^\ast_L(\mathcal{M}_1, \Gamma)$ by the supremum norm on $\mathrm{C}^k_b(\mathcal{M}_1, \Gamma)$.

It turns out that the canonical chain maps
\begin{align*}
\theta^k\colon \mathrm{C}^k_b(\widetilde{M}, \Gamma)\longrightarrow &\mathrm{C}^k_b(\mathcal{M}_1, \Gamma)\\
c\longmapsto &\left\{(\mu_0, ..., \mu_k)\longmapsto\int_{\widetilde{M}^{k+1}}c(y_0, ..., y_k)\mu_0\otimes...\otimes\mu_k\right\}
\end{align*}
induce an isometric isomorphism in cohomology $H_b^\ast(M, \mathbb{R})\cong H^\ast_L(\mathcal{M}_1, \Gamma)$, that we will also denote by $\theta^\ast$:
\begin{lemma}[\cite{besson-courtois-gallot}, Lemma 3.21]\label{isometric isomorphism}
Every continuous $\Gamma$-equivariant map $\psi\colon \widetilde{M}\rightarrow\mathcal{M}_1$ induces an isometric isomorphism $\psi^\ast\colon H^\ast_L(\mathcal{M}_1, \Gamma)\rightarrow H_b^\ast(M, \mathbb{R})$ with inverse map $\theta^\ast$.
\end{lemma}
\begin{remark}
The exact same proof as in \cite{besson-courtois-gallot} works for complete manifolds as well.
\end{remark}

Notice that to every $\Gamma$-equivariant map $\Phi^s\colon\widetilde{M}\rightarrow S^\infty$ from Definition \ref{def:Phi_s}, one can associate the composition $\psi^s=J\circ\Phi^s\colon\widetilde{M}\rightarrow\mathcal{M}_1$. It is $\Gamma$-equivariant as well: indeed, let $A$ be a Borel set of $\widetilde{M}$, let $\gamma\in \Gamma$ and $x\in \widetilde{M}$. Then
\begin{align*}
J\circ\Phi^s(\gamma x)(A)&=J(\gamma\Phi^s(x))(A)=\int_A(\gamma\Phi^s(x))^2(y)\widetilde{\mu}(y)\\
&=\int_A\Phi^s(x)^2(\gamma^{-1}y)\widetilde{\mu}(y)=\int_{\gamma^{-1}A}\Phi^s(x)^2(z)\widetilde{\mu}(z)\\
&=J(\Phi^s(x))(\gamma^{-1}A)=\gamma^\ast (J\circ\Phi^s(x))(A).
\end{align*}

Now we can prove statement (\ref{statement 1}): let $\beta\in H^k(M, \mathbb{R})$. If $\|\beta\|_\infty$ is infinite, there is nothing to show, except if $\delta=0$. But $\delta$ is the limit of $s$, $\delta/2<s\leq (\delta+1)/2$, and so $s\cdot\|\beta\|_\infty=\infty$ for all $\delta/2<s\leq (\delta+1)/2$. Thus we can make sense of the expression $\delta\cdot \|\beta\|_\infty$ and the inequality holds.

If $\|\beta\|_\infty$ is finite, $\beta$ has a bounded representative, and it defines a class in $H^k_b(M, \mathbb{R})$ that we still denote by $\beta$.
Via $\theta^k$ it is associated to a class $\theta^k(\beta)=[\overline{a}]\in H^\ast_L(\mathcal{M}_1, \Gamma)$ that has the same sup norm. Note that every representative $\overline{a}\in \mathrm{C}^k_b(\mathcal{M}_1, \Gamma)$ can be extended by linearity to an alternating $(k+1)$-form $a$ on $\mathcal {M}$. We associate to it
the differential form $\widetilde{a}$ as in Lemma \ref{alternating and differential form}. It is closed as $\overline{a}$ is a cocycle. According to Lemma \ref{isometric isomorphism}, the above $\Gamma$-equivariant map $\psi^s=J\circ\Phi^s$ induces an isometric inverse for $\theta$, hence $(J\circ\Phi^s)^\ast\widetilde{a}$ is a closed differential $k$-form on $\widetilde{M}$ representing the class $\beta$. We compute:
\begin{align*}
\mathrm{comass}(\beta)&\leq\mathrm{comass}((J\circ\Phi^s)^\ast\widetilde{a})\\
& \leq |\mathrm{Jac}_k(J\circ\Phi^s)|\mathrm{comass}(\widetilde{a})\\
&\leq 2^k\left(\frac{s^2}{k}\right)^{k/2}k!\cdot\|\overline{a}\|_\infty.
\end{align*}
This is true for all representatives $\overline{a}$ of $\theta^k(\beta)$, so we have $$\mathrm{comass}(\beta)\leq 2^k\left(\frac{s^2}{k}\right)^{k/2}k!\cdot\|\theta^k(\beta)\|_\infty=2^k\left(\frac{s^2}{k}\right)^{k/2}k!\cdot\|\beta\|_\infty$$ by isometry of $\theta$. Letting $s$ tend to $\delta/2$, we obtain the required result.

To prove statement (\ref{statement 2}), we proceed as follows.
For every map $\psi^s$ as above and for every $\alpha\in H_k(M, \mathbb{R})$, we define 
$$\mathrm{mass}(\psi^s(\alpha))=\sup\left\{\int_{\psi_k^s(\alpha)}\beta \mid \beta \mbox{ closed differential $k$-form on }\mathcal{M}, \mathrm{comass}(\beta)\leq 1\right\}.$$
By $\psi_k^s(\alpha)$ in the above definition we mean the following: let $a$ be a smooth singular cycle representing $\alpha$. Lift it to a singular chain in $\mathrm{C}_k(\widetilde{M}, \mathbb{R})$. Take its image in $\mathrm{C}_k(\mathcal{M},\mathbb{R})$ under the induced map $\psi_k^s$ and integrate over the resulting chain. By Stokes' theorem, as $\beta$ is closed, the result does not depend on the particular choice of $a$, nor on its lift to $\widetilde{M}$.

We want to show
\begin{equation}\label{Gromov inequality}
\|\alpha\|_1\leq k!\cdot\mathrm{mass}(\psi^s(\alpha)).
\end{equation}
Statement (\ref{statement 2}) will then follow from the following computation.
By taking the supremum over the appropriate $k$-forms $\beta$, we obtain:
\begin{align*}
\mathrm{mass}(J\circ\Phi^s(\alpha))&=\sup_{\beta,\, \mathrm{comass}(\beta)\leq 1}\int_{(J\circ\Phi^s)_k(\alpha)}\beta\\
&=\sup_{\beta,\, \mathrm{comass}(\beta)\leq 1} \int_{\alpha}(J\circ\Phi^s)^\ast\beta\\
&\leq \sup_{\beta,\, \mathrm{comass}(\beta)\leq 1}\mathrm{comass}((J\circ\Phi^s)^\ast\beta)\cdot\mathrm{mass}(\alpha)\\
&\leq \sup_{\beta,\, \mathrm{comass}(\beta)\leq 1} |\mathrm{Jac}_k(J\circ\Phi^s)|\cdot\mathrm{comass}(\beta)\cdot\mathrm{mass}(\alpha)\\
&\leq |\mathrm{Jac}_k(J)|\cdot|\mathrm{Jac}_k(\Phi^s)|\cdot\mathrm{mass}(\alpha)\\
&\leq 2^k\left(\frac{s^2}{k}\right)^{k/2}\mathrm{mass}(\alpha),
\end{align*}
where the last inequality uses Lemmas \ref{lem:k-Jac} and \ref{map J 2-lipschitz}. Combining with the inequality \eqref{Gromov inequality} and letting $s$ tend to $\delta/2$ yields the claim.

So we need to prove \eqref{Gromov inequality}. We proceed as follows. For every $\psi=\psi^s$ as above we have, by Lemma \ref{isometric isomorphism} and \eqref{duality Frigerio},
\begin{align*}
\|\alpha\|_1&=\sup\left\{\langle\beta, \alpha\rangle\mid \beta\in H^k_b(M, \mathbb{R}), \|\beta\|_\infty\leq 1\right\}\\
&=\sup\left\{\langle\psi^k\theta^k(\beta), \alpha\rangle\mid \beta\in H^k_b(M, \mathbb{R}), \|\beta\|_\infty\leq 1\right\}\\
&=\sup\left\{\langle\theta^k(\beta), \psi_k(\alpha)\rangle\mid \beta\in H^k_b(M, \mathbb{R}), \|\beta\|_\infty\leq 1\right\}.
\end{align*}
We want to estimate this quantity efficiently. Recall that if $c\in \mathrm{C}^k_b(\widetilde{M}, \Gamma)$ is a representative of the class $\beta$, then $\theta^k(c)$ is the restriction to $\mathcal{M}_1$ of the alternating form 
$$(\mu_0, ..., \mu_k)\longmapsto \int_{\widetilde{M}^{k+1}}c(y_0, ..., y_k)\mu_0\otimes...\otimes\mu_k$$ defined on $\mathcal{M}$, linear in each variable.

To the alternating $(k+1)$-form $a=\theta^k(c)$ on $\mathcal{M}$ we associate the differential $k$-form $\widetilde{a}$ on $\mathcal{M}_1$ as in Lemma \ref{alternating and differential form}. By this lemma, we have that
$$\mathrm{comass}(\widetilde{a})=k!\cdot\|\bar{a}\|_\infty,$$
where $\bar{a}$ denotes the restriction of $a$ to $\mathcal{M}_1$.
We thus can continue the computation:
\begin{align*}
\|\alpha\|_1&=\sup\left\{\langle\theta^k(\beta), \psi_k(\alpha)\rangle\mid \beta\in H^k_b(M, \mathbb{R}), \|\beta\|_\infty\leq 1\right\}\\
&= \sup\left\{\langle \overline{a}, \psi_k(\alpha)\rangle\mid [\overline{a}]=\theta^k(\beta)\in H^k_L(\mathcal{M}_1, \Gamma), \|\overline{a}\|_\infty\leq 1\right\}\\
&=\sup\left\{\int_{\psi_k(\alpha)}\widetilde{a}\mid [\overline{a}]=\theta^k(\beta)\in H^k_L(\mathcal{M}_1, \Gamma), \|\overline{a}\|_\infty\leq 1\right\}\\
&\leq\sup\left\{\int_{\psi_k(\alpha)}\omega\mid \omega \text{ closed  differential } k\text{-form on }\mathcal{M}, \mathrm{comass}(\omega)\leq k!\right\}\\
&= k!\cdot\sup\left\{\int_{\psi_k(\alpha)}\omega\mid \omega \text{ closed differential } k\text{-form on }\mathcal{M}, \mathrm{comass}(\omega)\leq 1\right\}\\
&=k!\cdot\mathrm{mass}(\psi(\alpha)),
\end{align*}
where the third equality is obtained thanks to Lemma \ref{alternating and differential form}. This shows \eqref{Gromov inequality} and hence finishes the proof of the theorem.
\end{proof}

\begin{remark}
The inequality of statement (\ref{statement 1}) can prove interesting in cases of classes where the sup norm is known: it will then give an upper bound on the comass of the associated differential forms.
The sup norm of a bounded cohomology class has been computed in several cases: for the volume class of hyperbolic manifolds, of course \cites{Thurston_1979, Gromov_1982}, but also for the Euler class of surface bundles over surfaces \cite{morita-article}, for the Euler class of flat bundles \cite{bucher-monod}, for the volume class of manifolds covered by $\mathbb{H}^2\times\mathbb{H}^2$ \cite{bucher}, for Hilbert modular surfaces \cite{Loeh_Sauer_2009}, for the K\"ahler class of a Lie group of Hermitian type \cite{Burger_Iozzi_Wienhard_2010}. There are upper and lower bounds for the volume class of the complex hyperbolic plane \cite{pieters}, and an upper bound for the volume class of surface bundles over surfaces \cite{bucher-bundles}.
\end{remark}

\bibliographystyle{siam}
\bibliography{bibliographie}
\end{document}